\documentclass{amsart}
\usepackage{amsmath}

\usepackage[psamsfonts]{amssymb}
\usepackage[all]{xy}
\usepackage{graphicx}
\usepackage[T1]{fontenc}
\usepackage[bitstream-charter]{mathdesign}
\usepackage{hyperref}
\newtheorem{theorem}{Theorem}[section]
\newtheorem{prop}[theorem]{Proposition}
\newtheorem{lemma}[theorem]{Lemma}
\newtheorem{cor}[theorem]{Corollary}

\theoremstyle{remark}
\newtheorem{dfn}[theorem]{Definition}
\newtheorem{remark}[theorem]{Remark}

\def\co{\colon\thinspace}

\def\ep{\epsilon}

\def\R{\mathbb{R}}

\def\C{\mathbb{C}}

\DeclareMathOperator{\osc}{osc}

\begin{document}

\title{Observations on the Hofer distance between closed subsets}
\author{Michael Usher}
\address{Department of Mathematics\\University of Georgia\\Athens, GA 30602\\usher@math.uga.edu}

\begin{abstract}
We prove the elementary but  surprising fact that the Hofer distance between two closed subsets of a symplectic manifold can be expressed in terms of the restrictions of Hamiltonians to one of the subsets; this helps explain certain energy-capacity inequalities that appeared recently in \cite{BM} and \cite{HLS}.  We also build on \cite{U} to obtain new vanishing results for the Hofer distance between subsets, applicable for instance to singular analytic subvarieties of K\"ahler manifolds.
\end{abstract}

\maketitle

This note uses rather elementary arguments to deduce some results about the Hofer distance between closed subsets, defined as the infimal Hofer norm of a Hamiltonian diffeomorphism that maps one subset to the other.  In the first section we give an alternative formula (Theorem \ref{main}) for this distance, which helps explain some seemingly-unexpectedly-strong versions of energy-capacity inequalities that appeared recently in \cite{BM} and \cite{HLS}, and indeed shows that all energy-capacity inequalities can be expressed in a similar strengthened form.  The second section contains new results about the rigid locus defined in \cite{U}, in particular connecting it to the Poisson bracket in Corollary \ref{raco}, and uses these to expand the class of subsets on whose orbits the Hofer distance is known to vanish identically.  Specifically this vanishing is established for all non-Lagrangian, half-dimensional submanifolds (Corollary \ref{halfdim}) and all analytic subvarieties (including singular ones) in K\"ahler manifolds (Theorem \ref{subvar}).

\section{Restricting the Hamiltonian}

There exists a rich history of results in symplectic topology asserting that, in order for a Hamiltonian diffeomorphism $\phi$ of a symplectic manifold $(M,\omega)$ to behave in a certain way with respect to a subset $A$ of $M$, the Hofer norm $\|\phi\|_{H}$ of $\phi$ must exceed some positive lower bound.  Here $\|\phi\|_{H}$ is by definition the infimal value of $\int_{0}^{1} \left(\max_M H(t,\cdot)-\min_M H(t,\cdot)\right)dt$ among smooth compactly supported functions $H\co [0,1]\times M\to\R$ having time-one map $\phi_{H}^{1}$ equal to $\phi$.  Indeed, the original proofs that $\|\cdot\|_{H}$ is nondegenerate \cite{Ho},\cite{LM} proceed by proving that, for $A$ equal to a closed Darboux ball in $M$, there is a number $c_A >0$ such that any Hamiltonian diffeomorphism $\phi$ such that $\phi(A)\cap A=\varnothing$ must have $\|\phi\|_{H}\geq c_A$; if $(M,\omega)$ is geometrically bounded \cite{Ch98} establishes a similar bound with $A$ instead equal to a compact Lagrangian submanifold of $M$, generalizing an earlier result of \cite{P}.  Along similar lines,   if $A$ is a compact Lagrangian submanifold of a tame symplectic manifold $(M,\omega)$ and if $U$ is either an open set intersecting $A$ or a compact Lagrangian submanifold that intersects $A$ transversely (and nontrivially) then there is a number $c_{A,U} > 0$ such that one has the bound $\|\phi\|_{H}\geq c_{A,U}$ whenever $\phi(A)\cap \bar{U}=\varnothing$. (This is \cite[Theorem 4.9]{U}; see also \cite[Corollary 3.7]{BC06}, \cite[Theorem J]{FOOO}, and \cite{Cha} for results which cover less general situations but have stronger bounds $c_{A,U}$.)

Recently, similar results to some of those above have appeared in \cite[Theorem 1.5(ii)]{BM} and in \cite[Lemma 9, citing \cite{LR}]{HLS}, but with a surprising twist.\footnote{We might also mention the results \cite[Lemma 2.1, Theorem 2.17(vi)]{MVZ} about Lagrangian spectral invariants, which in retrospect could  be seen as anticipating this phenomenon.}  For certain rather specific classes of symplectic manifolds $(M,\omega)$ and Lagrangian submanifolds $A$, for any open set $U$ intersecting $A$ these authors produce a positive constant $c_{A,U}$ which serves as a lower bound not only for the Hofer norm but also for the apparently-smaller quantity \begin{equation}\label{relint} \int_{0}^{1}\left(\max_A H(t,\cdot)-\min_A H(t,\cdot)\right)dt \end{equation} whenever the time-one map of $H\co [0,1]\times M\to \R$ disjoins $A$ from $\bar{U}$.  This appears counterintuitive: at time, say, $0.5$ one would expect the values of $H(0.5,\cdot)$ along $\phi_{H}^{0.5}(A)$ to be more relevant to the question of whether the Hamiltonian isotopy $\{\phi_{H}^{t}\}$ generated by $H$ moves $A$ out of $\bar{U}$ than the values of $H(0.5,\cdot)$ along $A$, yet it is the latter that contributes to (\ref{relint}).  The fact that the maximum and minimum in (\ref{relint}) can be taken over $A$ rather than $M$ is consequential: it plays a key role in the proof of the main result of \cite{HLS} on the $C^0$-rigidity of coisotropic submanifolds.

In this section we give a simple explanation for these results which give estimates for (\ref{relint}) instead of only for the Hofer norm: they do not, as might first appear, represent some new mysterious action-at-a-distance phenomenon in symplectic topology; rather, by means of elementary considerations about the relationships between Hamiltonians and their time-one maps we will see that the sorts of Hofer norm bounds described above immediately imply identical bounds on the quantity (\ref{relint}).\footnote{This is not to say that the aforementioned results in \cite{BM} and \cite{HLS} can be deduced from our argument together with prior results: even just as bounds on the Hofer norm, their lower bounds $c_{A,U}$ are in some cases larger than those given by other methods.}  In particular all of the bounds described in the first paragraph of this section can be combined with   Theorem \ref{main} below to yield bounds on (\ref{relint}) in the style of \cite{BM},\cite{HLS}.  

We now establish some basic notational conventions and definitions.  Throughout the paper, for a smooth manifold $P$ (possibly with boundary) we denote by $C^{\infty}_{0}(P)$ the set of smooth, compactly supported real-valued functions on $P$.

Let $(M,\omega)$ be a symplectic manifold without boundary. If $H\in C^{\infty}_{0}([0,1]\times M)$, for each $t\in [0,1]$  we let $H_t=H(t,\cdot)$ and let $X_{H_t}$ be the vector field obeying $\omega(\cdot,X_{H_t}) = dH_t$.  The Hamiltonian isotopy $\{\phi_{H}^{t}\}_{t\in [0,1]}$ is then characterized by the properties that $\phi_{H}^{0}=1_M$ and $\frac{d\phi_{H}^{t}}{dt} = X_{H_t}\circ\phi_{H}^{t}$.  As usual we denote by $Ham(M,\omega)$ the group consisting of those diffeomorphisms $\phi$ such that there exists $H\in C^{\infty}_{0}([0,1]\times M)$ with $\phi_{H}^{1}=\phi$.

For any closed subset $B\subset M$ and any $F\in C^{\infty}_{0}(M)$ write \[ \osc_B F = \max_B F - \min_B F\]  The Hofer norm on $Ham(M,\omega)$ is then defined by  \[ \|\phi\|_H = \inf\left\{\left.\int_{0}^{1}\osc_M H_t dt \right| \phi_{H}^{1}=\phi\right\} \]  

For the rest of this section fix a closed subset $A\subset M$, and let \[ \mathcal{L}(A) = \{\phi(A)|\phi\in Ham(M,\omega)\} \] denote the orbit of $A$ under the Hamiltonian diffeomorphism group.  We may then define $\delta\co \mathcal{L}(A)\times\mathcal{L}(A)\to\R$ by setting, for $A_0,A_1\in \mathcal{L}(A)$, $\delta(A_0,A_1) = \inf\{\|\phi\|_H |\phi(A_0)=A_1\}$, or equivalently (and more suggestively for our coming results) \[ \delta(A_0,A_1) = \inf\left\{\left. \int_{0}^{1}\osc_M H_t dt\right| \phi_{H}^{1}(A_0) = A_1\right\}. \]  It is easy to see that $\delta$ is a pseudometric on $\mathcal{L}(A)$ which is invariant under the action of $Ham(M,\omega)$.  In the case that $A$ is a Lagrangian submanifold the study of this pseudometric dates back at least to \cite{Oh} and \cite{Ch00}, and in the latter paper it is shown that if $(M,\omega)$ is geometrically bounded and $A$ is a compact Lagrangian submanifold then $\delta$ is nondegenerate.  See \cite{U} (and also the following section) for results about the behavior of $\delta$ when $A$ may not be Lagrangian.

We first prove the following simple lemma:

\begin{lemma}\label{cutoff}
Given $H\in C^{\infty}_{0}([0,1]\times M)$ and a closed subset $A\subset M$ there is $K\in C^{\infty}_{0}([0,1]\times M)$ such that:
\begin{itemize}
\item[(i)] $\phi_{H}^{t}(A)=\phi_{K}^{t}(A)$ for all $t$.
\item[(ii)] $\osc_{\phi_{H}^{t}(A)}H_t = \osc_{\phi_{H}^{t}(A)}K_t$ for all $t$.
\item[(iii)] For all $t\in [0,1]$ there is $a\in A$ such that $K(t,\phi_{K}^{t}(a))=0$.
\end{itemize}
\end{lemma}

\begin{proof} The proof splits into three cases depending on whether $A$ and $M$ are compact.  

If $A$ is noncompact then the fact that the support of $H$ is compact implies that $H(t,\cdot)|_{\phi_{H}^{t}(A)}$ takes the value $0$ for all $t$, so we can simply take $K=H$.

Assuming from now on that $A$ is compact, choose an arbitrary $a_0\in A$ and define $f\co [0,1]\to \R$ by $f(t)=H(t,\phi_{H}^{t}(a_0))$.  If $M$ is compact then the lemma will hold with $K(t,m)=H(t,m)-f(t)$.  If $M$ is not compact then this latter function might not be compactly supported, but since $A$ is compact we can find $\beta\in C^{\infty}_{0}(M)$ such that $\beta=1$ on a neighborhood of $\cup_{t}\{t\}\times \phi_{H}^{t}(A)$.  Then the lemma will hold with $K(t,m)=\beta(m)(H(t,m)-f(t))$, since the Hamiltonian vector fields of $H$ and $K$ coincide along $\cup_{t}\{t\}\times \phi_{H}^{t}(A)$.
\end{proof}

The following is well-known:

\begin{prop}\label{movingosc} For $A_0,A_1\in\mathcal{L}(A)$ we have \[ \delta(A_0,A_1) = \inf\left\{\left.\int_{0}^{1}\osc_{\phi_{H}^{t}(A_0)} H_t dt\right|\phi_{H}^{1}(A_0)=A_1\right\}.\]
\end{prop}

\begin{proof}
Choose any $H\in C^{\infty}_{0}([0,1]\times M)$ such that $\phi_{H}^{1}(A_0)=A_1$, and let $\ep > 0$.  Let $K\in C_{0}^{\infty}([0,1]\times M)$ be as in the previous lemma, applied with $A=A_0$.  Define $\Phi\co [0,1]\times M\to [0,1]\times M$ by $\Phi(t,m) = (t,\phi_{K}^{t}(m))$ and let $\Lambda = \Phi([0,1]\times A_0)$.  Choose a smooth function $\chi\co [0,1]\times M\to [0,1]$ such that $\chi$ is identically equal to $1$ on a neighborhood of the compact set $\Lambda\cap supp(K)$ and such that $\chi(t,m) = 0$ at all $(t,m)\in [0,1]\times M$ such that $K(t,m)\notin (\min_{\phi_{K}^{t}(A_0)}K_t-\ep/2,\max_{\phi_{K}^{t}(A_0)}K_t+\ep/2)$.  

Now let $K'=\chi K$.   Since for each $t$ we have $\min_{\phi_{K}^{t}(A_0)}K_t\leq 0\leq \max_{\phi_{K}^{t}(A_0)}K_t$ we see that, for all $t$, $\osc_{M}K'_{t} < \osc_{\phi_{K}^{t}(A_0)} K_t + \ep = \osc_{\phi_{H}^{t}(A_0)}H_t +\ep$. The fact that $K'$ coincides with $K$ on a neighborhood of $\Lambda$ readily implies that $\phi_{K'}^{t}(A_0)=\phi_{K}^{t}(A_0)$ for all $t$, and in particular that $\phi_{K'}^{1}(A_0)=A_1$.  Thus \[ \delta(A_0,A_1) \leq \inf\left\{\left.\int_{0}^{1}\osc_{\phi_{H}^{t}(A_0)} H_t dt\right|\phi_{H}^{1}(A_0)=A_1\right\} + \ep.\]  Since $\ep$ is arbitrary, we obtain the inequality ``$\leq$'' in the statement of the proposition, while of course the inequality ``$\geq$'' is trivial.  
\end{proof}



The main result of this section shows that, instead of taking the oscillation over the time-dependent (and $H$-dependent) closed set $\phi_{H}^{t}(A_0)$ as in Proposition \ref{movingosc}, we can simply take it over $A_0$:

\begin{theorem}\label{main}
For $A_0,A_1\in \mathcal{L}(A)$, we have \[ \delta(A_0,A_1) = \inf\left\{\left.\int_{0}^{1}\osc_{A_0} H_t dt\right|\phi_{H}^{1}(A_0)=A_1\right\} \]
\end{theorem}

\begin{proof}
The plan of the proof is to show that, for any $H\in C^{\infty}_{0}([0,1]\times M)$, there exists $K\in C^{\infty}_{0}([0,1]\times M)$ having the properties that \begin{equation}\label{kprops} \phi_{K}^{1}=\phi_{H}^{1}\qquad \mbox{and}\qquad \int_{0}^{1}\osc_{\phi_{K}^{t}(A_0)}K_t dt = \int_{0}^{1}\osc_{A_0}H_t dt.\end{equation}  In view of Proposition \ref{movingosc} this will obviously imply the inequality ``$\leq$`` in the statement of the theorem, while the inequality ``$\geq$'' just follows from the fact that $\osc_M H_t\geq \osc_{A_0}H_t$.

For a general $G\in C^{\infty}_{0}([0,1]\times M)$, consider the two functions \[ \bar{G}(t,m) = -G(t,\phi_{G}^{t}(m))\qquad \widehat{G}(t,m) = -G(1-t,m) \]

A standard calculation shows that $\bar{G}$ generates the Hamiltonian isotopy $\phi_{\bar{G}}^{t}=(\phi_{G}^{t})^{-1}$.  Meanwhile, $\widehat{G}$ is designed to have the property that a map $\gamma\co [0,1]\to M$ obeys $\gamma'(t) = X_{G_t}(\gamma(t))$ if and only if the time-reversed map $\widehat{\gamma}(t)=\gamma(1-t)$ obeys $\widehat{\gamma}'(t) = X_{\widehat{G}_t}(\widehat{\gamma}(t))$.  In other words, time-one Hamiltonian flowlines for $\widehat{G}$ are precisely time-reversals of time-one flowlines of $G$; at the level of isotopies this yields \[ \phi_{\widehat{G}}^{t} = \phi_{G}^{1-t}\circ(\phi_{G}^{1})^{-1} \]   In particular we have \[ \phi_{\bar{G}}^{1}=\phi_{\widehat{G}}^{1}=(\phi_{G}^{1})^{-1} \]

With this said, given $H\co [0,1]\times M\to\R$ we now produce the function $K\co [0,1]\times M\to\R$ promised in the first paragraph of the proof: \[ K=\overline{(\widehat{H})}\qquad \mbox{i.e.,}\qquad K(t,m) = -\widehat{H}(t,\phi_{\widehat{H}}^{t}(m)) = H\left(1-t,\phi_{H}^{1-t}((\phi_{H}^{1})^{-1}(m))\right)\]

We can quickly verify that the two properties in (\ref{kprops}) are satisfied: first of all, \[ \phi_{K}^{1}=\phi_{\overline{(\widehat{H})}}^{1}=(\phi_{\widehat{H}}^{1})^{-1}=((\phi_{H}^{1})^{-1})^{-1}=\phi_{H}^{1}.\]  Meanwhile, we have $\phi_{K}^{t}=(\phi_{\widehat{H}}^{t})^{-1}$ and so, for $(t,m)\in [0,1]\times M$, \[
K(t,\phi_{K}^{t}(m)) = -\widehat{H}(t,\phi_{\widehat{H}}^{t}(\phi_{K}^{t}(m)))=-\widehat{H}(t,m)=H(1-t,m). \]  
From this we see immediately that, for all $t$, \[ \osc_{\phi_{K}^{t}(A_0)} K_t = \osc_{A_0}H_{1-t} \] and hence \[ \int_{0}^{1}\osc_{\phi_{K}^{t}(A_0)}K_t dt = \int_{0}^{1}\osc_{A_0} H_{1-t}dt = \int_{0}^{1}\osc_{A_0}H_t dt,\] proving the second part of (\ref{kprops}) and hence the theorem.
\end{proof}

\begin{remark}
In cases where the Hamiltonian $H$ is time-independent simply setting $K=H$ in the above proof will of course lead to a Hamiltonian obeying (\ref{kprops}), in view of the conservation of energy property $H\circ\phi_{H}^{t}=H$.  In this situation one has $\widehat{H}=\overline{H}$, and so the Hamiltonian produced by our proof is indeed just $H$.  However in the time-dependent case $\widehat{H}$ and $\overline{H}$ will generally be distinct and the Hamiltonian $K$ in the proof will generate a different isotopy from the identity to $\phi_{H}^{1}$ than does $H$.
\end{remark}

\begin{remark} 
Since $\delta$ is symmetric and since $\osc_{A_1}H_t = \osc_{A_1}\hat{H}_{1-t}$, it follows from Theorem \ref{main} that we also have \[ \delta(A_0,A_1)=\inf\left\{\left.\int_{0}^{1}\osc_{A_1}H_t dt\right|\phi_{H}^{1}(A_0)=A_1\right\} \]  One can also prove this directly in the style of the above proof, by setting $K$ equal to $\widehat{(\overline{H})}$ instead of $\overline{(\widehat{H})}$ and observing that one then has $K(t,\phi_{K}^{t}(m)) = H(1-t,\phi_{H}^{1}(m))$ for all $(t,m)\in [0,1]\times M$.
\end{remark}

To connect this to the sorts of estimates described in at the beginning of this section, recall that the \emph{displacement energy} of the closed set $A$ is by definition \[ e(A) = \inf\left\{\|\phi\|_H\left|\phi(A)\cap A=\varnothing\right.\right\}.\]  For another subset $U\subset M$ (presumably intersecting $A$) we likewise define \[ e(A,U) = \inf\left\{\|\phi\|_H\left| \phi(A)\cap \bar{U}=\varnothing\right.\right\}.\]  As originally formulated, the results described in the first paragraph of this section (and many others like them) are lower bounds for $e(A)$ or $e(A,U)$ for various classes of $A$ and $U$.

\begin{cor} We have \begin{align*} e(A) &= \inf\left\{\left.\int_{0}^{1}\osc_A H_t dt\right|\phi_{H}^{1}(A)\cap A=\varnothing\right\} \quad \mbox{and}
\\ e(A,U) &= \inf\left\{\left.\int_{0}^{1}\osc_A H_t dt\right|\phi_{H}^{1}(A)\cap \bar{U}=\varnothing\right\}  \end{align*}
\end{cor}

\begin{proof} Since $A$ is assumed to be closed we have by definition $e(A)=e(A,A)$, so the first equation is a special case of the second.  For the second, simply note that, as an easy consequence of the definitions, \[ e(A,U) = \inf\left\{\delta(A,A')|A'\in\mathcal{L}(A),\,A'\cap \bar{U}=\varnothing\right\} \] and apply Theorem \ref{main}.
\end{proof}

The estimates in \cite{BM},\cite{HLS} that motivated this section were lower bounds for the right-hand side in the above corollary; we thus see that any of the numerous methods for estimating $e(A,U)$ in fact yields a similar estimate for this right-hand side.

\section{New properties of the rigid locus} \label{delta}

An immediate consequence of Theorem \ref{main} is that, for any closed subset $A\subset M$, if a function $H\in C^{\infty}_{0}([0,1]\times M)$ obeys $H|_{[0,1]\times A}=0$, then $\delta(A,\phi_{H}^{1}(A)) = 0$.\footnote{Of course this is not surprising in the special case that $A$ is a coisotropic submanifold, since then the hypothesis implies that $\phi_{H}^{1}(A)=A$.}  We will obtain below in Proposition \ref{irav} a strengthening of this result, in preparation for which we now recall some terminology from \cite{U}.

Again fixing a closed subset $A\subset M$, we write \[ \bar{\Sigma}_A = \{\phi\in Ham(M,\omega)|\delta(A,\phi(A)) = 0\}.\] (The notation refers to the fact that this is the closure of the stabilizer $\Sigma_A$ of $A$ with respect to the Hofer topology on $Ham(M,\omega)$; in particular $\bar{\Sigma}_A$ is a subgroup of $Ham(M,\omega)$, see \cite[Proposition 2.2]{U}.)  The \emph{rigid locus} of $A$ is then defined to be the set \[ R_A = \bigcap_{\phi\in\bar{\Sigma}_A}\phi^{-1}(A) \] So obviously $R_A\subset A$ (take $\phi=1_M$). It is easy to see that if $R_A=A$ then $\delta$ is nondegenerate on $\mathcal{L}(A)$.  A less obvious fact (originally proven as \cite[Lemma 4.2(iii)]{U}; this is also a special case of Proposition \ref{flex} below) is that if $R_A=\varnothing$ then $\delta$ vanishes identically on $\mathcal{L}(A)$.

Our main results in this section are strong new restrictions on the structure of the rigid locus $R_A$ (Corollaries \ref{raco} and \ref{notsmall}) which are then applied in Corollary \ref{halfdim} and Theorem \ref{subvar} to obtain new classes of subsets $A$ for which it always holds that $R_A=\varnothing$ and hence that the pseudometric $\delta$ vanishes identically.

For any open subset $U\subset M$ let $Ham_U$ denote the subgroup of $Ham(M,\omega)$ consisting of Hamiltonian diffeomorphisms generated by (extensions by zero of) Hamiltonians $H\in C^{\infty}_{0}([0,1]\times U)$.

\begin{prop}\label{flex} For any closed set $A\subset M$ we have
\[ Ham_{M\setminus R_A}\subset \bar{\Sigma}_A. \]
\end{prop}

\begin{proof}
The proof is very similar to that of \cite[Lemma 4.2(iii)]{U}, which concerns the case that $R_A=\varnothing$.  Given $x\in M\setminus R_A$ we may find $\psi_x\in \bar{\Sigma}_A$ such that $\psi_x(x)\notin A$; since $A$ is closed we can then find a neighborhood $U_x$ of $x$ such that $\psi_x(U_x)\cap A=\varnothing$.  Then \[ \psi_x Ham_{U_x}\psi_{x}^{-1} = Ham_{\psi_x(U_x)}\subset  \bar{\Sigma}_A \]  (indeed every element of $Ham_{\psi_x(U_x)}$ preserves $A$).
So since $\bar{\Sigma}_A$ is a subgroup of $Ham(M,\omega)$ and $\psi_x\in\bar{\Sigma}_A$ it follows that $Ham_{U_x}\leq \bar{\Sigma}_A$.

We have thus found an open cover $\{U_x\}_{x\in M\setminus R_A}$ of $M\setminus R_A$ with the property that each $Ham_{U_x}$ is contained in $\bar{\Sigma}_A$.  But the fragmentation lemma of \cite[III.3.2]{Ba} (applied to the symplectic manifold $M\setminus R_A$, which is an open subset of $M$) asserts that $Ham_{M\setminus R_A}$ is generated by $\cup_x Ham_{U_x}$, so that $Ham_{M\setminus R_A}\subset \bar{\Sigma}_A$.
\end{proof}

The following shows that a Hamiltonian which only vanishes on $R_A$, not necessarily on all of $A$, continues to have the property that its flow sends $A$ to sets which lie a distance zero away from $A$. 
\begin{prop} \label{irav}
Suppose that $H\in C^{\infty}_{0}([0,1]\times M)$ has $H|_{[0,1]\times R_A} = 0$.  Then $\phi_{H}^{s}\in \bar{\Sigma}_A$ for all $s\in [0,1]$.

\end{prop}

\begin{proof} Where $H^s(t,m) = sH(st,m)$ for $s\in [0,1]$, we have $\phi_{H^s}^{1}=\phi_{H}^{s}$, so since $H^s|_{[0,1]\times R_A}=0$ whenever $H|_{[0,1]\times R_A}=0$ it suffices to prove the result for $s=1$.

So assume that $H|_{[0,1]\times R_A}=0$ and for any natural number $n$ let $f_n\co \R\to\R$ be a smooth, nondecreasing function such that $f_n(s)=s$ for $|s|\geq \frac{1}{n}$ and $f_n(s)=0$ for $|s|<\frac{1}{2n}$.  Then $\|f_n\circ H - H\|_{C^0}\leq \frac{1}{n}$, and so $\phi_{f_n\circ H}^{1}\to \phi_{H}^{1}$ as $n\to\infty$ with respect to the Hofer topology on $Ham(M,\omega)$.  But $f_n\circ H$ vanishes on the neighborhood $\{|H|<\frac{1}{2n}\}$ of $[0,1]\times R_A$ and has support contained in the (compact) support of $H$, so $\phi_{f_n\circ H}^{1}\in Ham_{M\setminus R_A}$.  Thus by Proposition \ref{flex}, $\phi_{f_n\circ H}^{1}\in \bar{\Sigma}_A$ for all $n$.  But $\bar{\Sigma}_A$ is closed in the Hofer topology, so it follows that $\phi_{H}^{1}\in \bar{\Sigma}_A$ also.
\end{proof}

For the rest of the paper we will focus on autonomous Hamiltonians $H\in C^{\infty}_{0}(M)$.  We continue to denote by $\phi_{H}^{t}$ the Hamiltonian flow of the function on $[0,1]\times M$ defined by $(t,m)\mapsto H(m)$.

For a general closed subset $B\subset M$ we denote 
\[ I_B = \left\{H\in C^{\infty}_{0}(M)\left|H|_B=0\right.\right\}.\]

\begin{cor}\label{raco} Where $\{F,G\}=\omega(X_F,X_G)$ is the Poisson bracket, the subset $I_{R_A}\subset C^{\infty}_{0}(M)$ is closed under $\{\cdot,\cdot\}$.
\end{cor}

\begin{proof}
Let $F,G\in I_{R_A}$.   It follows immediately from the definition of $R_A$ (and the fact that $\bar{\Sigma}_A$ is a subgroup of $Ham(M,\omega)$) that $R_A$ is preserved by all elements of $\bar{\Sigma}_A$, so since Proposition \ref{irav} asserts that $\phi_{F}^{t}\in \bar{\Sigma}_A$ for all $t$, we have $\phi_{F}^{t}(R_A)=R_A$ for all $t$.  So the fact that $G\in I_{R_A}$ implies that $G\circ\phi_{F}^{t}$ vanishes identically on $R_A$ for all $t$.  Thus for $x\in R_A$ we have \[ \{F,G\}(x)=(dG)_x(X_F) = \left.\frac{d}{dt}G(\phi_{F}^{t}(x))\right|_{t=0} = 0, \] \emph{i.e.} $\{F,G\}\in I_{R_A}$.
\end{proof}

\begin{remark}
Corollary \ref{raco} imposes rather strong restrictions on the possible geometry of the rigid locus $R_A$ of any closed subset.  
It is a standard (and easily checked) fact that if $B\subset M$ is a submanifold then $I_B$ is closed under $\{\cdot,\cdot\}$ if and only if $B$ is coisotropic.  Thus if the rigid locus is a submanifold then it is coisotropic.

Also we recover the fact (\cite[Corollary 4.5]{U}) that, if $A\subset M$ is a submanifold, $\delta$ can be nondegenerate on $\mathcal{L}(A)$ only if $A$ is coisotropic: indeed if $\delta$ were nondegenerate we would have $R_A=A$, and as just noted, given that $R_A=A$ is a submanifold $R_A$ is coisotropic.
\end{remark}

\begin{cor} \label{ball2}
Let $x\in R_A$, and suppose that $F_1,\ldots,F_k\in I_{R_A}$ have the property that $(dF_1)_x,\ldots,(dF_k)_x\in T^{*}_{x}M$ are linearly independent.  Then the map \begin{align*} \psi\co \R^k &\to M \\ (a_1,\ldots,a_k) &\mapsto \phi_{\sum a_i F_i}^{1}(x) \end{align*} has image contained in $R_A$, and restricts to a sufficiently small ball around the origin as an embedding.
\end{cor}

\begin{proof} The linearization of $\psi$ at $\vec{0}\in \R^k$ sends the standard basis vectors $e_1,\ldots,e_k$ to $(X_{F_1})_x,\ldots,(X_{F_k})_x$, and these are linearly independent by the assumption that $(dF_1)_x,\ldots,(dF_k)_x\in T^{*}_{x}M$ are linearly independent.  Thus the restriction of $\psi$ to a suitably small neighborhood of $\vec{0}$ is an immersion, and its restriction to a smaller neighborhood is an embedding.

Because each function $\sum a_i F_i$ belongs to $I_{R_A}$, by Proposition \ref{irav} we have $\phi_{\sum a_i F_i}^{1}\in\bar{\Sigma}_A$ for each $\vec{a}$.  Since $R_A$ is preserved by the action of any element of $\bar{\Sigma}_A$, and since $x\in R_A$, for each $\vec{a}\in \R^k$ it follows that $\psi(\vec{a}) = 
\phi_{\sum a_i F_i}^{1}(x)\in R_A$.
\end{proof}

The following resolves a question that was raised in \cite[Section 4.2]{U}.

\begin{cor}\label{notsmall}
Let $A\subset M$ be any closed subset such that $R_A\neq \varnothing$ and suppose that $N\subset M$ is any submanifold which is closed as a subset.  If $\dim N<\frac{1}{2}\dim M$ then $N$ does not contain $R_A$, while if $N$ is connected and $\dim N=\frac{1}{2}\dim M$ then $N$ does not properly contain $R_A$.
\end{cor}

\begin{proof}  Suppose to the contrary that we have $R_A\subset N$ where $N$ is as in the statement of the corollary.  Let $k=\dim M-\dim N$, and choose any $x\in R_A$.  We can then obtain functions $F_1,\ldots,F_k$ as in Corollary \ref{ball2} by taking a coordinate chart around $x$ in which $N$ appears as $\{\vec{0}\}\times \R^{\dim M - k}$ and then multiplying the first $k$ coordinate functions by a cutoff function which is equal to $1$ on a small neighborhood of $x$. Hence Corollary \ref{ball2} gives an embedding $B^{k}(\delta)\hookrightarrow R_A\subset N$ of a small $k$-dimensional ball $B^{k}(\delta)$, with image containing $x$.  

If $\dim N <\frac{1}{2}\dim M$ this immediately gives a contradiction since in this case $k>\dim N$ but we have just embedded a $k$-dimensional ball into $N$.  In the remaining case that $\dim N=\frac{1}{2}\dim M$ (so $\dim N=k$) the $k$-dimensional ball that we have embedded into $R_A\subset N$ necessarily contains a neighborhood of $x$ in $N$.  Since $x\in R_A$ was chosen arbitrarily this proves that $R_A$ is open in $N$.  But as an immediate consequence of its definition, $R_A$ is also closed.  So since $N$ is assumed connected and by hypothesis $R_A\neq \varnothing$, it must be that $R_A=N$.
\end{proof}

\begin{cor} \label{halfdim}
Let $A\subset M$ be a submanifold of dimension $\frac{1}{2}\dim M$ which is connected and closed as a subset.  Then the pseudometric $\delta$ on $\mathcal{L}(A)$ either vanishes identically or is nondegenerate.  In particular if $A$ is not Lagrangian then $\delta$ vanishes identically.
\end{cor}

\begin{proof}
Taking $N=A$ in Corollary \ref{notsmall}, since one always has $R_A\subset A$ we see that the hypothesis implies that 
either $R_A=A$ or $R_A=\varnothing$, \emph{i.e.} (by \cite[Lemma 4.2]{U}) either $\delta$ is nondegenerate or $\delta$ vanishes identically.  If $A$ is not Lagrangian (equivalently, not coisotropic) then the first alternative cannot hold (by \cite[Corollary 4.5]{U}, or Corollary \ref{raco} above).
\end{proof}

\begin{remark} As mentioned earlier, Chekanov showed in \cite{Ch00} that if $(M,\omega)$ is geometrically bounded and $A$ is a compact Lagrangian submanifold then $\delta$ is nondegenerate.  The same paper contains an example (attributed to Sikorav) of a compact Lagrangian submanifold of a non-geometrically-bounded symplectic manifold for which $\delta$ vanishes identically.
\end{remark}

\begin{remark} 
For submanifolds of codimension strictly between $1$ and $\frac{1}{2}\dim M$ it is possible for $\delta$ to neither be nondegenerate nor vanish identically, as explained in \cite[Remark 1.5]{U}.
\end{remark}

\begin{remark}\label{sublag}  Corollary \ref{notsmall} also evidently implies (again taking $N=A$) that if $A\subset M$ is a connected closed submanifold of dimension at most $\frac{1}{2}\dim M$ and if $B\subset A$ is any proper closed subset then $R_B=\varnothing$ and so $\delta$ vanishes identically on $\mathcal{L}(B)$.
\end{remark}


\subsection{Subvarieties}

In this subsection we prove Theorem \ref{subvar}, asserting that $\delta$ vanishes identically on $\mathcal{L}(A)$ whenever $A$ is a (possibly singular) complex analytic subvariety of a K\"ahler manifold.  Accordingly let $(M,\omega,J)$ be a K\"ahler manifold  (so $\omega$ is a symplectic form and $J$ is an $\omega$-compatible integrable almost complex structure).  We then obtain a Riemannian metric $g\co TM\times_M TM\to\R$ defined by $g(v,w) = \omega(v,Jw)=\omega(-Jv,w)$.  Define maps $\theta_{\omega},\theta_g\co TM\to T^*M$ by $\theta_{\omega}(v)=\omega(v,\cdot)$ and likewise $\theta_g(v)=g(v,\cdot)$. Thus $\theta_{g}=-\theta_{\omega}\circ J$.  Since $\omega$ is nondegenerate, $\theta_{\omega}$ and $\theta_{g}$ are invertible, and we see that $\theta_{g}^{-1}=J\circ\theta_{\omega}^{-1}$.  Define the dual metric $g^*\co T^*M\times_M T^*M\to \R$ by $g^{*}(\alpha,\beta)=g(\theta_{g}^{-1}(\alpha),\theta_{g}^{-1}(\beta))$.  Of course by the definition of $\theta_g$ we have $g^{*}(\alpha,\beta)=\alpha(\theta_{g}^{-1}(\beta))$.  

\begin{prop}\label{holpois}
Let $U\subset M$ be an open subset and let $f\co U\to\C$ be a holomorphic function, written as $f=u+iv$ where $u,v\co U\to \R$.  Then the Poisson bracket of $u$ and $v$ is given everywhere on $U$ by \[ \{u,v\}=g^{*}(du,du) = g^{*}(dv,dv) \]
\end{prop}

\begin{proof}
In our present notation the Hamiltonian vector field of $u$ is given by $X_u=-\theta_{\omega}^{-1}(du)$.  So 
\[
\{u,v\} = \omega(X_u,X_v) = dv(X_u) = -dv(\theta_{\omega}^{-1}(du)) = dv(J\theta_{g}^{-1}(du)) 
\]  But the Cauchy--Riemann equation for the holomorphic function $f$ amounts to the statement that $dv\circ J = du$, so the above gives $\{u,v\}=du(\theta_{g}^{-1}(du)) = g^{*}(du,du)$.  

Meanwhile since $J$ is an isometry with respect to $g$, the adjoint of $J$ is an isometry with respect to $g^{*}$, and so the fact that $dv\circ J=du$ implies that $g^{*}(dv,dv)=g^{*}(du,du)$.
\end{proof}

\begin{dfn}
Let $A$ be a closed subset of the K\"ahler manifold $(M,\omega,J)$ and let $x\in X$.  A \textbf{holomorphic reducing chart} $(U,V,\psi,f)$  for $A$ around $x$ consists of the following data:
\begin{itemize} \item A connected open neighborhood $U\subset M$ of $x$ having compact closure.
\item An open set $V\subset\C^n$, and a holomorphic chart $\psi\co V\to M$ such that $\bar{U}\subset \psi(V)$.
\item A holomorphic function $f\co V\to \C$ such that $f|_{\psi^{-1}(R_A)}=0$ where $R_A$ is the rigid locus of $A$
\end{itemize}
\end{dfn}

\begin{prop}\label{redind} If $(U,V,\psi,f)$ is a holomorphic reducing chart for $A$ around $x$ then there is an open subset $V'\subset V$ which contains $\psi^{-1}(\bar{U})$ such that, for each $j=1,\ldots,n$, $\left(U,V',\psi|_{V'},\frac{\partial f}{\partial z_j}\right)$ is also a holomorphic reducing chart for $A$ around $x$.
\end{prop}

\begin{proof}
Let $\beta\co M\to[0,1]$ be a smooth function which is identically equal to $1$ on some open set $U'$ containing $\bar{U}$ but whose support is compact and contained in $\psi(V)$.  Define $F\co M\to \C$ to be equal to $\beta\cdot(f\circ\psi^{-1})$ on $\psi(V)$ and to $0$ on $M\setminus\psi(V)$, and let $u$ and $v$ be, respectively, the real and imaginary parts of $F$.  Since $f|_{\psi^{-1}(R_A)}=0$, we have $u|_{R_A}=v|_{R_A}=0$.
So by Corollary \ref{raco}, $\{u,v\}|_{R_A}=0$. 

Now the restriction of $F=u+iv$ to $U'$ is holomorphic, so by Proposition \ref{holpois} we have $\{u,v\}|_{U'}=g^{*}(du,du)=g^{*}(dv,dv)$.  So since $\{u,v\}|_{R_A}=0$ we obtain $dF=0$ at each point of $R_A\cap U'$.  Letting $V'=\psi^{-1}(U')$, we have $F|_{V'}=f\circ (\psi|_{V'})^{-1}$, so we deduce that $df=0$ at each point of $(\psi|_{V'})^{-1}(R_A)$.  The fact that each $\left(U,V',\psi|_{V'},\frac{\partial f}{\partial z_j}\right)$ is a holomorphic reducing chart then follows immediately from the definition.
\end{proof}

\begin{cor}\label{redcor} If $(U,V,\psi,f)$ is a holomorphic reducing chart for $A$ around $x$ such that $f$ is not identically zero on $\psi^{-1}(U)$, then $x\notin R_A$.
\end{cor}

\begin{proof} By Proposition \ref{redind} and induction, for each multi-index $\alpha$ we obtain a holomorphic reducing chart for $A$ around $x$ of the form $\left(U,V_{\alpha},\psi|_{V_{\alpha}},\frac{\partial^{|\alpha|}f}{\partial z^{\alpha}}\right)$ where $V_{\alpha}$ is a neighborhood of $\bar{U}$.  If it were the case that $x\in R_A$ we would then obtain that $f$ vanishes to infinite order at $\psi^{-1}(x)$.  But since $\psi^{-1}(U)$ is connected and $f$ is holomorphic this implies that $f|_{\psi^{-1}(U)}$ is identically zero.
\end{proof}

\begin{theorem}\label{subvar}
Let $A$ be a complex analytic subvariety of positive codimension in a K\"ahler manifold $(M,\omega,J)$, or more generally any closed subset of a complex analytic subvariety of positive codimension.  Then $R_A=\varnothing$, and so $\delta$ vanishes identically on $\mathcal{L}(A)$.
\end{theorem}

\begin{proof} By definition, $M$ is covered by the images of holomorphic charts $\psi_{\alpha}\co V_{\alpha}\to M$ each having the property that $\psi_{\alpha}^{-1}(A)$ is contained in the zero locus of some holomorphic function $f_{\alpha}\co V_{\alpha}\to \C$ that is not identically zero on any nonempty open subset.  Since $R_A\subset A$, then, if $U$ is any connected open subset whose closure is compact and contained in $\psi_{\alpha}(V_{\alpha})$ the tuple $(U,V_{\alpha},\psi_{\alpha},f_{\alpha})$ is a holomorphic reducing chart for $A$ around any point of $U$.  Such a $U$ can be found for any $x\in \psi_{\alpha}(V_{\alpha})$, so Corollary \ref{redcor} shows that $\psi_{\alpha}(V_{\alpha})\cap R_A=\varnothing$.
So since the various $\psi_{\alpha}(V_{\alpha})$ cover $M$, $R_A=\varnothing$.
\end{proof}

\begin{remark}
As Remark \ref{sublag} and the proof of Theorem \ref{subvar} illustrate, arguments that show that a point $x$ does not lie in the rigid locus of some subset $A$ often also show that $x\notin R_B$ whenever $B$ is a closed subset of $A$.  It seems natural to expect that one always has the inclusion $R_B\subset R_A$ whenever $B\subset A$, but I do not know a proof of this statement.
\end{remark}

\section*{Acknowledgements} I am grateful to Sobhan Seyfaddini for comments on an earlier draft of the paper.  Part of this work was motivated by talks at the workshop on $C^0$ Symplectic Topology and Dynamical Systems at IBS in Pohang in January 2014.  The work was partially supported by NSF grant DMS-1105700.

\end{document}